\documentclass[12pt]{amsart}

\usepackage{amssymb,latexsym}
\usepackage{amsmath,amsfonts,amsthm,amscd,amsxtra,enumerate}
\usepackage{mathtools}
\usepackage[all]{xy}
\usepackage{amsmath,latexsym,amssymb, enumerate, amsthm, epsfig, hyperref} %supertabular,xypic}

\usepackage[margin=1in]{geometry}
\usepackage{epsfig}
\usepackage{[;''''''''''''youngtab}
\usepackage{young}
\usepackage[vcentermath]{youngtab}
\usepackage{young}
\usepackage{youngtab}
\usepackage{ytableau}

\newtheorem{theorem}{Theorem}[section]
\newtheorem{lemma}[theorem]{Lemma}
\newtheorem{corollary}[theorem]{Corollary}
\newtheorem{definition}[theorem]{Definition}
\newtheorem{example}[theorem]{Example}
\newtheorem{remark}[theorem]{Remark}
\newtheorem{proposition}[theorem]{Proposition}

\newtheorem{setup}[theorem]{Setup}

\newcommand{\sk}{{\ensuremath{\sf k }}}
\newcommand{\m}{\ensuremath{\mathfrak m}}

\DeclareMathOperator{\soc}{soc}
\DeclareMathOperator{\im}{im}

\DeclareMathOperator{\Hom}{Hom}
\DeclareMathOperator{\pdim}{pdim}
\DeclareMathOperator{\depth}{depth}

\begin{document}

\title{\textbf{Embeddings of Canonical Modules and Resolutions of Connected Sums}}

\author[Ela Celikbas]{Ela Celikbas}
\address{Department of Mathematics, West Virginia University, Morgantown, WV 26506.}
\email{ela.celikbas@math.wvu.edu}

\author[Jai Laxmi]{Jai Laxmi}
\address{Department of Mathematics, I.I.T. Bombay, Powai, Mumbai 400076.}
\email{jailaxmi@math.iitb.ac.in}

\author{Jerzy Weyman}
\address{Department of Mathematics, University of Connecticut, Storrs, CT 06269.}
\email{jerzy.weyman@uconn.edu}

\date{\today}


%05A10

\thanks{Jerzy Weyman was partially supported by NSF grant DMS-1400740.}

\begin{abstract}
For an ideal  $I_{m,n}$ generated by all square-free monomials of degree $m$ in a polynomial ring $R$ with $n$ variables, we obtain a specific embedding of a canonical module of $R/I_{m,n}$ to $R/I_{m,n}$ itself. The construction of this explicit embedding depends on a minimal free $R$-resolution of an ideal generated by $I_{m,n}$. Using this embedding, we give a resolution of connected sums of several  copies of certain Artinian $\sk$-algebras where $\sk$ is a field.

\end{abstract}

\maketitle

\section{Introduction}

For a Cohen-Macaulay ring $S$ with a canonical module $\omega_S$, it is well-known that, if $S$ is generically Gorenstein (e.g.
$S$ is reduced), then $\omega_S$ can be identified with an ideal of $S$, that is, $\omega_S$ embeds into $S$; see, for example \cite[3.3.18]{BH}. In this paper we give an explicit construction of such an embedding for a certain ring. More precisely, if $R$ is a polynomial ring in $n$ variables over a field $\sk$, $I_{m,n}$ is the ideal of $R$ generated by all square-free monomials of degree $m$ and $\omega_{R/I_{m,n}}$ is the canonical module of $R/I_{m,n}$, then, in Theorem \ref{inj}, we establish an explicit standard graded embedding of $\omega_{R/I_{m,n}}$ into $R/I_{m,n}$. Our motivation for this study comes from obtaining minimal free resolutions of \emph{connected sums} of Gorenstein rings. As given in \cite{ACJY}, a connected sum of several Gorenstein rings $S_i$ is a Gorenstein ring $S$  that is a special quotient of the fiber product (pullback) of $S_i$'s. Indeed, as a consequence of our argument, we give a construction of a resolution of a connected sum of several copies of $S_i:=\sk[x]/(x^{e_i+1})$ over a field $\sk$; see Corollary \ref{Connectedsum}.

In order to construct a specific embedding from $\omega_{R/I_{m,n}}$ to $R/I_{m,n}$, we use generators of the $R/I_{m,n}$-module $\Hom_R(\omega_{R/I_{m,n}},R/I_{m,n})$. In section \ref{HOM}, we give a set of generators of $\Hom_R(R/I_{m,n},R/I_{m,n})$ in Theorem \ref{Mingens}. Moreover, as an immediate result of Theorem \ref{Mingens}, we get a presentation of  $\Hom_R(\omega_{R/I_{m,n}},R/I_{m,n})$. Section \ref{POINC} deals with the computation of Hilbert-Poincar\'{e} functions of $R/I_{m,n}$ and $\omega_{R/I_{m,n}}$. 

The main result of this paper is Theorem \ref{inj} which gives a specific standard graded embedding of a canonical module 
$$\psi:=\psi_{m,n}:\omega_{R/I_{m,n}}\longrightarrow R/I_{m,n}.$$
 In Corollary \ref{omg},  the image of $\psi_{m,n}$ is identified with an ideal of $R/I_{m,n}$ generated by maximal minors of a certain Vandermonde-like matrix $D$. We use Theorem \ref{inj} and Corollary \ref{omg} to get a resolution of a Gorenstein ring obtained from an embedding of a canonical module of $R/I_{m,n}$ in Corollary \ref{Gor1}. 

In section \ref{sec6} we specialize to $m=2$. In this case, in Theorem \ref{m=2}, we give another,  $\mathbb{N}^n$-graded embedding of a canonical module of $R/I_{2,n}$ into the ring $R/I_{2,n}$. As a corollary of this theorem, a canonical module of $R/I_{2,n}$ is identified with an $\mathbb{N}^n$-graded ideal of $R/I_{2,n}$. 
The mapping cone of the map of free resolutions over $R$ covering the embedding $\omega_{R/I_{m,n}}\longrightarrow R/I_{m,n}$ gives a minimal free resolution of the connected sum of algebras $S_i:=\sk[x]/(x^{e_i+1})$.

Section \ref{Prelim} contains known results regarding the main tools used in the rest of the paper including the definition of connected sums, resolutions of the ideals generated by square-free monomials of a given degree and of corresponding Stanley-Reisner rings.

\section{Preliminaries}\label{Prelim}
\subsection{Notation}
\begin{enumerate}[a)]
\item For a positive integer $n$, let $[n]=\{1,\ldots,n\}$. If $\sigma\subset [n]$, then $|\sigma|$ denotes the number of elements contained in $\sigma$.
\item Let $R=\sk[x_1,\ldots,x_n]$ be a polynomial ring in $n$ variables over a field $\sk$ with $x_1>\ldots >x_n $. We order the monomials in $R$ with graded lexicographic order. 
\item Let $m$ and $n$ be positive integers with $m\leq n$, then $I_{m,n}$ denotes an ideal generated by all square-free monomials of degree $m$ in $n$ variables. Furthermore, $\omega_{R/I_{m,n}}$ denotes a canonical module of $R/I_{m,n}$. 

\item For an $R$-module $M$, $\ell(M)$ and $\mu(M)$ denote the length and the minimal number of generators of $M$, respectively.
\item For a commutative Noetherian ring $T$, $\dim(T)$ denotes the Krull dimension of $T$.
\item Let $M=\oplus_{i\geq 0} M_i$ be a graded $R$-module. The Hilbert-Poincar$\mathbf{\acute{\text{e}}}$  function of $M$ is the formal power series 
$H_{M}(t)=\sum_{i\geq 0}\ell(M_i)t^i.$   
\item Let $(T,\m,\sk)$ be an Artinian local ring. Then the {\it socle} of $T$ is $\soc(T) = (0:_T\m)$.
\item For a Noetherian local ring $T$ and a $T$-module $M$, a finite presentation of $M$ is an exact sequence $T^{\oplus m}\rightarrow T^{\oplus n}\rightarrow M\rightarrow 0$ with $m,n$ positive integers.
\end{enumerate}

\subsection{Connected Sums}
\begin{definition}\label{CSdef}{\rm
Let $S_i=\sk[x_i]/(x_i^{e_i+1})$, $\soc(S_i)=(x_i^{e_i})$, and $J=\langle x_ix_j,x_i^{e_i}-x_1^{e_1}|1\leq i\leq n\rangle$ where $e_i\geq 1$ be the ideal in $R:=\sk[x_1,\ldots ,x_n]$ defining the connected sum $S_1\#_{\sk}\ldots \#_{\sk}S_n$ of the algebras $S_i$ (compare \cite{ACJY} for the definition of connected sums). Therefore we have
$$S_1\#_{\sk}\ldots \#_{\sk}S_n=R/J.$$ }
\end{definition}
\begin{remark}\label{CSrmk}
With notation in Definition \ref{CSdef}, $S_1\#_{\sk}\ldots\#_{\sk}S_n$ is Gorenstein by \cite{ACJY}. 
\end{remark}

\subsection{Specht Modules and  Free Resolution of the Ring $\bf {R/I_{m,n}}$}
We recall the definition of Specht module $S^{(p,1^q)}$ associated to a hook partition $(p,1^q)$ of $n$  where $p,q$ are nonnegative integers. We follow \cite[Section 7.4]{FW}. Let $n=p+q$ and let $\mathbb{S}_n$ be a symmetric group on $[n]$.  Let $(p,1^q)$ be a hook partition of $n$. An oriented column tabloid  of shape $(p,1^q)$ is filling of Young diagram of $(p,1^q)$  with positive integers  $1,2,\ldots, n$, with each number appearing once,  which is skew-symmetric in the first column and symmetric in the remaining rows.

The Specht module $S^{(p,1^q)}$ is the $\sk$-vector space generated by the equivalence classes $[T]$ of oriented column tabloids  of shape $(p,1^q)$ with entries in $[n]$ modulo the following relations: 
\begin{enumerate}[a)]
\item Alternating columns: $\sigma[T]={\rm sign}(\sigma)[T]$ for all $\sigma\in \mathbb{S}_n$ fixing the columns of $T$ (so in the case of hook, just permuting the numbers in the first column). 
\item Shuffling relations: $[T]=\sum[T']$, where sum is over all $T'$ acquired from $T$ by exchanging the element of the second column of $[T]$ with one of the element of the first column of $T$.
\end{enumerate}

We recall some facts about Specht modules associated to a hook partition $(p,1^q)$.
\begin{enumerate}[a)]
\item The Specht module $S^{(p,1^q)}$ is a $\sk$-vector space. By using hook length formula from \cite[Theorem 20.1]{JM}, we have 
$${\rm dim}(S^{(p,1^q)})=\binom{p+q-1}{q}.$$
\item The symmetric group $\mathbb{S}_n$ acts on $S^{(p,1^q)}$ by permuting the numbers in oriented column tabloids.
\item An oriented column tabloid of shape $(p,1^q)$ is called standard tableau of shape $(p,1^q)$ if the entries in each row and column are increasing (in the case of hooks it means the entries in the first column are increasing and the first entry in the first column is $1$). The equivalence classes of standard tableaux of shape $(p,1^q)$ form a $\sk$-basis of $S^{(p,1^q)}$. 
Let $SYT([p+q],(p,1^q))$ denote the set of standard tableaux of shape $(p,1^q)$ with entries  $1,2,\ldots ,p+q$ (each number appearing once).
\end{enumerate}

In the remaining part of this subsection, we state the results from \cite{GF}.
\begin{definition}{\rm
(a) Let $n$, $m$ and $i$ be integers. For $1\leq m\leq n$ and $0\leq k\leq n-m$,   
\begin{equation}\label{U}
U_k^{m,n}:={\rm Ind}^{\mathbb{S}_n}_{\mathbb{S}_{m+k}\times\mathbb{S}_{n-m-k}}
(S^{(m,1^k)}\otimes_{\sk}S^{(n-m-k)}).
\end{equation}
Here $S^{(n-m-k)}$ is the Specht module $S^{(p,1^q)}$ with $p:=n-m-k, q:=0$.
The right hand side of Equation \ref{U} is the $\sk[\mathbb{S}_n]$-module induced by the $\sk[\mathbb{S}_{m+k}\times\mathbb{S}_{n-m-k}]$-module $S^{(m,1^k)}\otimes_{\sk}S^{(n-m-k)}$. If any of the inequalities involving $n,m$, and $i$ are violated, then we set $U_i^{m,n}:=0$.\\
(b) A $\sk[\mathbb{S}_n]$-module $F_k^{m,n}$ is defined as
\begin{equation}
F_k^{m,n}:=U_k^{m,n}\otimes_{\sk}R(-m-k).
\end{equation}

}
\end{definition} 

\begin{remark}\label{Tabl}{\rm 
Let $n,m$, and $k$ be positive integers with $1\leq m\leq n$ and $0\leq k\leq n-m$. 
\begin{enumerate}[(a)]
\item The module $U_k^{m,n}$ is generated by the equivalence classes of oriented column tabloids of shape $(m,1^k)$, filled with numbers $1,2,\ldots, n$ without repetitions. Moreover, the equivalence classes of standard tableaux of shape $(m,1^k)$ form a $\sk$-basis of $U_k^{m,n}$. 
\item The module $F_k^{m,n}$ is a free $R$-module generated by the equivalence classes of oriented column tabloids of shape $(m,1^k)$, filled with numbers $1,2,\ldots, n$ without repetitions. 
\item The equivalence classes of standard tableaux of the shape $(m,1^k)$ with entries in $[n]$ (without repetitions)  form an $R$-basis of $F_k^{m,n}$ and  the rank of $F^{m,n}_k$ is $\beta_k:={\rm rank}(F_k^{m,n})=\binom{n}{m+k}\binom{m+k-1}{k}$.

\end{enumerate}
}
\end{remark}
 We define an $R$-linear map 
 $$\partial_k^{m,n} :F^{m,n}_k\longrightarrow F^{m,n}_{k-1}$$
 by setting
 $$\partial_k^{m,n}([T]):=\sum_{p=0}^{k}(-1)^{k-p}x_{i_p}[T\setminus i_p]$$
 where $[T]$ is an oriented column tabloid of shape $(m,1^k)$, and $T\setminus i_p$ is an oriented column tabloid of shape $(m, 1^{k-1})$ obtained from $T$ by omitting the number $i_p$ in position $p-1$ in the first column of $T$.

\begin{proposition}\label{Gal}{\rm
Let $n,m$ and $k$ be positive integers with $m\leq n$ and $0\leq k\leq n-m$. 
Then $(\mathbb{F}^{m,n}_{\bullet},\partial_\bullet^{m,n})$ is a complex of free $R$-modules which is a minimal free resolution of the $R$-module $R/I_{m,n}$. 
This complex is $\mathbb{S}_n$-equivariant, where $\mathbb{S}_n$ acts on $F^{m,n}_k$ diagonally (the action on $R$ just permutes the variables $x_i$).
}
\end{proposition}

\subsection{Simplicial Complex and Stanley-Reisner Rings}

\begin{definition}\cite[Definition 5.1.1]{BH} {\rm
Let $V=\{v_1,\ldots,v_n\}$ be a finite set. 
\begin{enumerate}
\item A non-empty set $\Delta$ of subsets of $V$ with the property that $\tau\in \Delta$ whenever $\tau\subset \sigma$ for some $\sigma\in \Delta$ is called a simplicial complex on the vertex set $V$.
The elements of $\Delta$ are called faces, and the dimension, $\dim \sigma$, of a face $\sigma$ is the number $|\sigma|-1$. The dimension of the simplicial complex $\Delta$ is $\dim(\Delta)=\max\{\dim \sigma: \sigma\in \Delta\}$.
\item Let $\sk$ be a field. The Stanley-Reisner ring of the complex $\Delta$ is the homogeneous $\sk$-algebra 
$\sk[\Delta]=\sk[x_1,\ldots,x_n]/I_{\Delta},$
where $I_{\Delta}$ is the ideal generated by all monomials $x_{i_1}\ldots x_{i_s}$ such that $\{v_{i_1},\ldots,v_{i_s}\}\not\in \Delta$. The Krull dimension of the Stanley-Reisner ring $\sk[\Delta]$ is $\dim(\Delta)+1$. 
\end{enumerate} } 
\end{definition}

\begin{lemma}
Let $I_{\Delta}=I_{m,n}$, then $\sk[\Delta]$ is Cohen-Macaulay.
\end{lemma}
\begin{proof}
If $I_{\Delta}=I_{m,n}$, then all monomials $x_{i_1}\ldots x_{i_{m-1}}\in I_{\Delta}$, so $\{v_{i_1},\ldots,v_{i_{m-1}}\}\not \in \Delta$. Hence, $\dim(\Delta)=m-2$. Then $\dim(\sk[\Delta])=m-1$. 

By Proposition \ref{Gal}, projective dimension of $\sk[\Delta]$ is $n-m+1$. By graded Auslander-Buchsbaum formula, $\depth(\sk[\Delta])=m-1$. Thus, $\sk[x_1,\ldots,x_n]/I_{m,n}$ is Cohen Macaulay.  
\end{proof}

\begin{remark}\label{MinOmg}
By Proposition \ref{Gal}, $(\mathbb{F}_{\bullet}^{m,n},\partial_\bullet^{m,n})$ is a minimal free resolution of $R/I_{m,n}$. Let $\mathbb{G}_{\bullet}^{m,n}=\Hom_R(\mathbb{F}_{\bullet}^{m,n},R)$ be the dual complex. Then $\mathbb{G}_{\bullet}^{m,n}$ is a minimal free $R$-resolution of $\omega_{R/I_{m,n}}$.   
\end{remark}

\section{Generators of Hom}\label{HOM}

The goal of this section is to find the generators (Theorem \ref{Mingens}) and a presentation (Corollary \ref{pres}) of the $R$-module $\Hom_R(\omega_{R/I_{m,n}},R/I_{m,n})$. We start with the example
$n=4, m=2$.

\begin{example}\label{ex1}{\rm
Let $R=\sk[x_1,x_2,x_3,x_4]$ and $I_{2,4}=\langle x_1x_2,x_1x_3,x_1x_4,x_2x_3,x_2x_4,x_3x_4\rangle$ be an ideal of $R$. Let $[T_{[4]\setminus{\{2\}}}]=\Bigg[\scriptsize\young(12,3,4)\Bigg]$, $[T_{[4]\setminus{\{3\}}}]=\Bigg[\scriptsize\young(13,2,4)\Bigg]$ and $[T_{[4]\setminus{\{4\}}}]=\Bigg[\scriptsize\young(14,2,3)\Bigg]$.
The formulas for differentials in the complex $\mathbb{F}^{4,2}_\bullet$ are
\begin{align*}
\partial_2^{2,4}([T_{[4]\setminus{\{2\}}}])&=x_1\bigg[\scriptsize\young(32,4)\bigg]-x_3\bigg[\scriptsize\young(12,4)\bigg]+x_4\bigg[\scriptsize\young(12,3)\bigg]\\
&=x_1\bigg[\scriptsize\young(23,4)\bigg]-x_1\bigg[\scriptsize\young(24,3)\bigg]-x_3\bigg[\scriptsize\young(12,4)\bigg]+x_4\bigg[\scriptsize\young(12,3)\bigg].\\
\partial_2^{2,4}([T_{[4]\setminus{\{3\}}}])&=x_1\bigg[\scriptsize\young(23,4)\bigg]-x_2\bigg[\scriptsize\young(13,4)\bigg]+x_4\bigg[\scriptsize\young(13,2)\bigg].\\
\partial_2^{2,4}([T_{[4]\setminus{\{4\}}}])&=x_1\bigg[\scriptsize\young(24,3)\bigg]-x_2\bigg[\scriptsize\young(14,3)\bigg]+x_3\bigg[\scriptsize\young(14,2)\bigg].\\
\end{align*}

Let $P$ be the matrix of $\partial_2^{2,4}$ with respect to the bases of standard tableaux in modules $F^{4,2}_2$ and $F^{4,2}_1$. 
$$P=\begin{bmatrix*} 0&0&x_4\\0&x_4&0\\ 0&0&-x_3\\ x_3&0&0\\ 0&-x_2&0\\ -x_2&0&0\\ 0&x_1&x_1\\ x_1&0&-x_1\end{bmatrix*}$$

Columns are listed in order $T_{[4]\setminus{\{4\}}}=\scriptsize\young(14,2,3)$, $T_{[4]\setminus{\{3\}}}=\scriptsize\young(13,2,4)$, and $T_{[4]\setminus{\{2\}}}=\scriptsize\young(12,3,4)$, and rows 

\smallskip
are listed in order $\scriptsize\young(12,3)$, $\scriptsize\young(13,2)$, $\scriptsize\young(12,4)$, $\scriptsize\young(14,2)$, $\scriptsize\young(13,4)$, 
$\scriptsize\young(14,3)$, $\scriptsize\young(23,4)$, and $\scriptsize\young(24,3)$.

\smallskip

Then the transpose of $P$, denoted by $P^T$, gives a matrix presentation of $\omega_{R/I_{2,4}}$. For $2\leq i\leq 4$, let $f_{\{i\}}:\omega_{R/I_{2,4}}\rightarrow R/I_{2,4}$ be defined as 
\begin{equation}\label{eqex1}
f_{\{i\}}([T_{[4]\setminus\{j\}}])=\begin{cases}x_i,\quad  \text{if}\;\; i=j\\ 0,\;\quad \text{if}\;\; i\neq j.\end{cases}
\end{equation}
In order to show that $f_{\{i\}}$ is well defined, it is enough to prove that $f_{\{i\}}$ satisfies the relations of  $P^T$. Note that the entry $x_i$ is missing in the column corresponding to $\partial_{2}^{2,4}([T_{[4]\setminus \{i\}}])$ in $P$, hence $ f_{\{i\}}$ satisfies the relations of $P^T$.

The tableau $[T_{[4]\setminus \{1\}}]=\Bigg[\scriptsize\young(21,3,4)\Bigg]$ is expressed in terms of standard tableaux such as $$\Bigg[\scriptsize\young(21,3,4)\Bigg]=\Bigg[\scriptsize\young(12,3,4)\Bigg]-\Bigg[\scriptsize\young(13,2,4)\Bigg]+\Bigg[\scriptsize\young(14,2,3)\Bigg].$$ Let 
\begin{equation}\label{eqex2}
f_{\{1\}}([T_{[4]\setminus \{2\}}])=-f_{\{1\}}([T_{[4]\setminus \{3\}}])=
f_{\{1\}}([T_{[4]\setminus \{4\}}])=x_1.
\end{equation}
 Since $\partial_{2}^{2,3}([T_{[4]\setminus \{1\}}])=\partial_{2}^{2,4}([T_{[4]\setminus \{2\}}])-\partial_2^{2,4}([T_{[4]\setminus \{3\}}])+\partial_2^{2,4}([T_{[4]\setminus \{4\}}])$, there is no term involving $x_1$ in $\partial_2^{2,3}([T_{[4]\setminus \{1\}}])$. Hence, $f_{\{1\}}$ is well defined.

Now suppose $\psi:\omega_{R/I_{2,4}}\rightarrow R/I_{2,4}$ satisfies $\psi([T_{[4]\setminus {\{i\}}}])=c+u$ for some $c\in \sk$, and $u\in \langle x_1,\ldots, x_n\rangle$. Then $\psi$ satisfies the relations of $P^T$, which implies, $cx_i=0$,  hence $c=0$. 
Then, taking into account relations given by the first six columns of $P^T$,  we can write
$$ \psi([T_{[4]\setminus {\{4\}}}])=\sum_{e\ge 1}a_1^{(e)}x_1^{(e)}+\sum_{e\ge 1}a_4^{(e)}x_4^{(e)},$$
$$ \psi([T_{[4]\setminus {\{3\}}}])=\sum_{e\ge 1}b_1^{(e)}x_1^{(e)}+\sum_{e\ge 1}b_3^{(e)}x_3^{(e)},$$
$$ \psi([T_{[4]\setminus {\{2\}}}])=\sum_{e\ge 1}c_1^{(e)}x_1^{(e)}+\sum_{e\ge 1}c_2^{(e)}x_2^{(e)},$$

where $a_i^{(e)}$, $b_i^{(e)}$, and $c_i^{(e)}$ are all in $\sk$.

Using the relations from last two columns of $P^T$,  we get $a_1^{(e)}=c_1^{(e)}  =-b_1^{(e)}$ for each $e$.
This means 
$$\psi =  (\sum_{e\ge 1}c_1^{(e)}x_1^{(e-1)})f_1 + (\sum_{e\ge 1}c_2^{(e)}x_2^{(e-1)})f_2 + (\sum_{e\ge 1}b_3^{(e)}x_3^{(e-1)})f_3 +(\sum_{e\ge 1}a_4^{(e)}x_4^{(e-1)})f_4.$$

 This shows that $\{f_{\{1\}},f_{\{2\}},f_{\{3\}},f_{\{4\}}\}$ is a minimal generating set of $\Hom_R(\omega_{R/I_{2,4}},R/I_{2,4})$.

}
\end{example}

In the light of the example above, the following theorem gives a general description of a minimal generating set of $\Hom_R(\omega_{R/I_{m,n}},R/I_{m,n})$.
\begin{theorem}\label{Mingens}
For $1<j_1<\ldots< j_{m-1}\leq n$, let $\Theta=\{j_1,\ldots,j_{m-1}\}$. Suppose 
$f_{j_1,\ldots,j_{m-1}}$ and $f_{1,j_2,\ldots,j_{m-1}}$ are maps from $\omega_{R/I_{m,n}}$ to $R/I_{m,n}$ defined as 

$f_{j_1,\ldots,j_{m-1}}([T_{[n]\setminus \Gamma }])=\begin{cases} x_{j_1}x_{j_2}\ldots x_{j_{m-1}},\quad \text{if}\;\Gamma=\Theta \\ 0 \;\;\qquad\qquad, \text{otherwise.} \end{cases}$ \;\;\;\; \;\;and 

\[
f_{1,j_2,\ldots,j_{m-1}}([T_{[n]\setminus \Gamma }])=\begin{cases} x_{1}x_{j_2}\ldots x_{j_{m-1}}, \quad \text{if}\; \Gamma=\Theta\;\;\text{or}\;\; \Gamma=(\Theta\setminus\{j_1\})\cup \{l\}\;\; \text{for}\;\; 1\neq l\in [n]\setminus \Theta,\\ 0\;\;\qquad\qquad, \text{otherwise.}
\end{cases}
\]
Then $\{f_{j_1,\ldots,j_{m-1}},f_{1,j_2,\ldots,j_{m-1}}: 1<j_1<\ldots<j_{m-1}\leq n\}$ is a minimal generating set of $\Hom_R(\omega_{R/I_{m,n}},R/I_{m,n})$.
\end{theorem}
\begin{proof}
First note that by Remark \ref{Tabl}, $\mathcal{B}_{k}:=\{[T]: T\in SYT((m,1^k),[n])\}$ is a basis of $F_k^{m,n}$. For $1=i_0<j_1<\ldots<j_{m-1}\leq n$ and $1=i_0<i_1<i_2<\ldots<i_{n-m}\leq n$, we set standard tableau of shape $(m,1^{n-m})$ as
\begin{equation}\label{tab}
\ytableausetup{centertableaux}
\ytableausetup
{mathmode, boxsize=2.2em}
[T_{[n]\setminus \{j_1,\ldots,j_{m-1}\}}]=[T_{1,i_1,\ldots, i_{n-m}}]=
\begin{ytableau}
\scriptstyle
i_0 & j_1 & j_2&\ldots & j_{m-1} \\
i_1 \\
i_2\\
\vdots\\
i_{n-m}
\end{ytableau} 
\end{equation}

By Proposition \ref{Gal}, we get the differential $\partial^{m,n}_{n-m}:F_{n-m}^{m,n}\rightarrow F_{n-m-1}^{m,n}$ as  

 \begin{equation}
 \partial^{m,n}_{n-m}([T_{1,i_1,\ldots,i_{n-m}}])=\sum_{k=0}^{n-m}(-1)^{n-m-k}x_{i_k}[T_{1,i_1,\ldots,i_{n-m}}\setminus i_k].
\end{equation}

Let $P$ be the matrix of $\partial_{n-m}^{m,n}$ with respect to the bases $\mathcal{B}_{n-m}$ and $\mathcal{B}_{n-m-1}$. Then $P^T$, the transpose of $P$, is a presentation of $\omega_{R/I_{m,n}}$ by Remark \ref{MinOmg}. 

In order to show that $f_{j_1,\ldots,j_{m-1}}$ and $f_{1,j_2,\ldots,j_{m-1}}$ are well defined, it is enough to see that $f_{j_1,\ldots,j_{m-1}}$ and $f_{1,j_2,\ldots,j_{m-1}}$ satisfy the relations of $P^T$. Since the column with respect to $\partial^{m,n}_{n-m}([T_{1,i_1,\ldots,i_{n-m}}])$ does not involve $x_{j_k}$, the corrensponding row in $P^T$ has no $x_{j_k}$ as well. Thus $f_{j_1,\ldots,j_{m-1}}$ satisfies the relations of $P^T$. A non-standard tableau $[T_{[n]\setminus \{1,j_2,\ldots,j_{m-1}\}}]$ can be expressed in terms of standard tableaux as 

$$[T_{[n]\setminus \{1,j_2,\ldots,j_{m-1}\}}]=[T_{[n]\setminus \{j_1,\ldots,j_{m-1}\}}]+\sum_{k=0}^{n-m}(-1)^k[T_{[n]\setminus \{i_k,j_2,\ldots,j_{m-1}\}}].$$

By direct computation one sees that column with respect to $\partial_{n-m}^{m,n}([T_{[n]\setminus \{1,j_2,\ldots,j_{m-1}\}}])$ does not involve $x_1$ and $x_{j_k}$ for $k=2,\ldots,m-1$. Therefore, $f_{\{1,j_2,\ldots,j_{m-1}\}}$ satisfies the relations of $P^T$, hence $f_{\{1,j_2,\ldots,j_{m-1}\}}$ is well defined.  

We now claim that $\{f_{\tau}:\tau\subset [n], |\tau|=m-1\}$ is a generating set of $\Hom_R(\omega_{R/I_{m,n}},R/I_{m,n})$. Let $\varphi\in\Hom_R (\omega_{R/I_{m,n}}, R/I_{m,n})$. For $1<l_1<\ldots<l_{m-1}\leq n$, let $\sigma=\{l_1,\ldots,l_{m-1}\}\subset [n]$.  Since $\varphi([T_{[n]\setminus\sigma }])\in R/I_{m,n}$, we can write $$\varphi([T_{[n]\setminus \sigma}])=\sum_{\tau=\{p_1,\ldots,p_{m-1}\}\subset [n]}a_{\tau} x_{p_1}^{n_{p_1}}\ldots x_{p_{m-1}}^{n_{p_{m-1}}}+\sum_{\tau=\{p_1,\ldots,p_k\}\subset [n], k<m-1}b_{\tau}x_{p_1}^{m_{p_1}}\ldots x_{p_k}^{m_{p_{k}}}$$
where $n_{p_k},m_{p_l}\geq 0$ and $a_{\tau},b_{\tau}\in \sk$. The fact that $\varphi$ satisfies the relations of $P^T$ implies $b_{\tau}=0$ and $a_{\tau}=0$ provided $\tau\neq \sigma$ or $\tau\neq\{1,l_2,\ldots,l_{m-1}\}$. Thus we get  
\begin{equation}\label{eqt}
\varphi([T_{[n]\setminus \sigma}])=c_{\sigma}f_{\sigma}([T_{[n]\setminus \sigma}])+c_{(\sigma\setminus \{l_1\})\cup \{1\}}f_{(\sigma\setminus \{l_1\})\cup\{1\}}([T_{[n]\setminus \sigma}])
\end{equation}
where $c_{\sigma}=a_{\sigma}x_{l_1}^{n_{l_1}-1}\ldots x_{l_{m-1}}^{n_{l_{m-1}-1}}$ and $c_{(\sigma\setminus \{l_1\})\cup \{1\}}=a_{(\sigma\setminus \{l_1\})\cup \{1\}}x_1^{n_1-1}x_{l_2}^{n_{l_2}-1}\ldots x_{l_{m-1}}^{n_{l_{m-1}-1}}$.

For every $\Gamma\subset [n]$ with $1\not\in \Gamma$ and $|\Gamma|=m-1$, the equivalence class of a standard tableau $[T_{[n]\setminus \Gamma}]$ is in $\mathcal{B}_{n-m}$. By Equation \ref{eqt}, for $c_{\tau}\in R/I_{m,n}$, we get $$\varphi([T_{[n]\setminus \Gamma}])=\sum_{\tau\subset [n],|\tau|=m-1}c_{\tau}f_{\tau}([T_{[n]\setminus\Gamma}]).$$

Since $[T_{[n]\setminus \Gamma}]\in \mathcal{B}_{n-m}$, we have $\varphi([T])=\sum_{\tau\subset [n],|\tau|=m-1}c_{\tau}f_{\tau}([T])$ for every oriented column tabloid $[T]$. Hence $\varphi\in \langle f_{\tau}:\tau\subset [n],|\tau|=m-1\rangle$. This proves that  $\langle f_{\tau}:\tau\subset [n],|\tau|=m-1\rangle$ is a generating set of $\Hom_R(\omega_{R/I_{m,n}},R/I_{m,n})$.

If  $\langle f_{\tau}:\tau\subset [n],|\tau|=m-1\rangle$ is not a minimal generating set, then for some $\Gamma\subset [n]$ with $|\Gamma|=m-1$, $f_{\Gamma}=\sum_{\tau\subset [n],|\tau|=m-1,\tau\neq \Gamma}a_{\tau}f_{\tau}$ where $a_{\tau}\in\sk$. Then for $1\in \gamma$ and $|\gamma\cap \Gamma|=m-2$,  $x_{\Gamma}=a_{\gamma}x_{\gamma}$ which is not possible. 
\end{proof}

As a consequence of Theorem \ref{Mingens}, we get a finite presentation as stated below.
\begin{corollary}\label{pres}
Let $S=R/I_{m,n}$, $\sigma \subset [n]$, and $|\sigma|=m-1$. Suppose $f_{\sigma}:\omega_{S}\rightarrow S$ is a map defined in Theorem \ref{Mingens}. Let $\mathcal{C}=\{e_{\sigma}:\sigma \subset [n]\}$ and $\mathcal{D}=\{p_{\sigma\cup\{i\}}:\sigma\subset [n],i\not\in \sigma\}$ be bases of $S^{\binom{n}{m-1}}$ and $S^{m\binom{n}{m}}$, respectively. Then 
\[S^{m\binom{n}{m}}\xrightarrow{\mu} S^{\binom{n}{m-1}}\xrightarrow{\phi} \Hom_R(\omega_{S},S)\rightarrow 0\]
with $\phi(e_{\sigma})=f_{\sigma}$ and $\mu(p_{\sigma\cup\{i\}})=x_ie_{\sigma}$ is a finite presentation of $\Hom_R(\omega_{S},S)$.
\end{corollary}
\begin{proof}
Observe that $f_{\sigma}:\omega_{R/I_{m,n}}\rightarrow R/I_{m,n}$ defined in Theorem \ref{Mingens} is an $R/I_{m,n}$-module homomorphism and $\Hom_R(\omega_{S},S)$ is generated by $\binom{n}{m-1}$ elements. Then we show that $\ker(\phi)=\langle x_ie_{\sigma}:\sigma\subset [n]\rangle$. Since there is a surjective map $\phi:S^{\binom{n}{m-1}}\rightarrow \Hom_R(\omega_{S},S)$ defined by $\phi(e_{\sigma})=f_{\sigma}$, by Theorem \ref{Mingens}, $x_if_{\sigma}=0$ for each $\sigma \subset [n]$ and $i\notin \sigma$. Thus $\phi(x_ie_{\sigma})=0$ and hence $x_ie_{\sigma} \in \ker(\phi)$. 

Assume to the contrary that we have a relation $\sum_{\sigma} a_\sigma f_\sigma =0$ with $a_\sigma$ being a polynomial depending only on the variables $x_i$ with $i\in \sigma$. We need to show that each $a_\sigma =0$. Let us fix a subset $\tau$ and let us apply the zero homomorphism to the tableau $T_{[n]\setminus \tau}$.
We get
$$ 0=a_\tau \prod_{i\in\tau}x_i+\sum_{\sigma\ne\tau} a_\sigma\prod_{i\in\sigma}x_i f_{\sigma} (T_{[n]\setminus \tau}).$$
But the first summand cannot cancel with any other summand since it is the only monomial containing precisely the variables from $\tau$. This shows that $a_{\tau}=0$. 
Therefore \[S^{m\binom{n}{m}}\xrightarrow{\mu} S^{\binom{n}{m-1}}\xrightarrow{\phi} \Hom_R(\omega_{S},S)\rightarrow 0\]
is an exact sequence.   
\end{proof}

\section{Hilbert-Poincar\'{e} function}\label{POINC}
In this section, we compute the Hilbert-Poincar$\mathbf{\acute{\text{e}}}$ functions of $R/I_{m,n}$ and $\omega_{R/I_{m,n}}$ by using Stanley-Reisner rings.

\begin{lemma}\label{Omega}
Let $R/I_{m,n}$ be a standard graded ring. The Hilbert-Poincar\'{e} function of $R/I_{m,n}$ is of the form
$${\displaystyle H_{R/I_{m,n}}(t)=\frac{\sum_{j=0}^{m-1}h_jt^j}{(1-t)^{m-1}}},~~ \text{where}~~ h_j=\dbinom{n-m+j}{j}.$$
Moreover, $H_{\omega_{R/I_{m,n}}}(t)=\dfrac{\sum_{i=0}^{m-1}\alpha_i t^i}{(1-t)^{m-1}}$, where $\alpha_i=\dbinom{n-i-1}{m-i-1}$.
\end{lemma}
\begin{proof}
Suppose $\Delta$ is a simplicial complex on the vertex set $V=\{v_1,\ldots,v_n\}$ such that $\{v_{i_1},\ldots,v_{i_m}\}\not\in \Delta$ for each $1\leq i_1<\ldots <i_m\leq n$. Then the Stanley-Reisner ring of the complex $\Delta$ is  the homogeneous $\sk$-algebra 
$$\sk[\Delta]=\sk[x_1,\ldots,x_n]/I_{m,n}$$
where $I_{n,m}$ is the ideal generated by all monomials of degree $m$.

Let $f_i$ denote the number of $i$-dimensional faces of $\Delta$. Then $f_{i-1}=\dbinom{n}{i}$ for $0\leq i\leq m-1$. By a known combinatorial identity, we get

$$\dbinom{n-m+j}{j}=\sum_{i=0}^{j}(-1)^{j-i}\dbinom{m-1-i}{j-i}\dbinom{n}{i}.$$ 
Then, by \cite[Lemma. 5.1.8]{BH}, we have $H_{R/I_{n,m}}(t)=\dfrac{\sum_{j=0}^{m-1}h_jt^j}{(1-t)^{m-1}}$, where $h_j=\dbinom{n-m+j}{j}$. Now one can see that
 $$H_{\omega_{R/I_{m,n}}}(t)=(-1)^{m-1}H_{R/I_{m,n}}(t^{-1})=\dfrac{\sum_{j=0}^{m-1}h_jt^{m-1-j}}{(1-t)^{m-1}}=\dfrac{\sum_{j=0}^{m-1}h_{m-1-j}t^{j}}{(1-t)^{m-1}}$$ by \cite[Corollary. 4.4.6]{BH}. Thus, for $\alpha_j=h_{m-1-j}$, we get $H_{\omega_{R/I_{m,n}}}(t)=\dfrac{\sum_{j=0}^{m-1}\alpha_j t^j}{(1-t)^{m-1}}$. 
\end{proof} 

Let us fix an $n$-tuple $(e_1, \ldots, e_n)$ of integers greater than $1$. Let $r_1, \ldots, r_n$ be defined as $r_i=e_1\ldots \hat{e}_i\ldots e_n$ where $\hat{e}_i$ denotes the missing term in the product, and $e=e_1\ldots e_n$. For $2\leq i\leq n$ we set $\deg(x_i)=r_i$ which makes $R$ an $\mathbb{N}^n$-graded ring.
Let $\tilde{H}_{M}(t)$ denote the Hilbert function of a module $M$ in this new grading. 

In the following remark, we give the Hilbert-Poincar\'{e} functions of $R/I_{2,n}$ and $R/L$.
 
\begin{remark}\label{Poin1}
Suppose $\deg(x_i)=r_i$ and $L_{2,n}=\langle x_i^{e_i}-x_n^{e_n}:1\leq i\leq n-1\rangle$, and $L=I_{2,n}+L_{2,n}$. Then $$\displaystyle{\tilde{H}_{R/I_{2,n}}(t)=1+\sum_{i=1}^n\dfrac{t^{r_i}}{1-t^{r_i}}}\;\; \text{and}\;\; \displaystyle{\tilde{H}_{R/L}(t)=1+\sum_{i=1}^n \dfrac{t^{r_i}-t^e}{1-t^{r_i}}}+t^e.$$
\end{remark}

\section{$\mathbb{N}$-Graded Embedding of A Canonical Module}\label{EMBD}

Throughout this section we fix $1\le m<n$ and the integers $\underline d =(d_1,\ldots, d_{m-1})$ satisfying $1<d_1<\ldots<d_{m-1}$.

 In this section, for each $\underline d$, an explicit embedding $\psi:=\psi(\underline d)$ of $\omega_{R/I_{m,n}}$ into $R/I_{m,n}$ is constructed in Theorem~\ref{inj}. We also prove that $\omega_{R/I_{m,n}}$ is identified with an $\mathbb{N}$-graded ideal of $R/I_{m,n}$ for each $1<d_1<\ldots<d_{m-1}$ where $d_i\in \mathbb{N}$. In order to do that, we order monomials  in graded lexicographic order and all initial ideals are taken with respect to that order. 

\begin{definition}
Let $I$ be an ideal of $R$ and $0\neq f\in R$. The initial ideal of $I$, denoted ${\rm in}(I)$, is defined as 
$${\rm in}(I)=\{{\rm in}(f)|f\in I\setminus \{0\}\}$$  
where ${\rm in}(f)$ is the largest monomial appearing in $f$.
\end{definition}

\begin{setup}\label{Setup}
Let $m-1\nmid {\rm char}(\sk)$. For $1< d_1<d_2<\ldots<d_{m-1}$, let $d=d_1+\ldots+d_{m-1}$. Consider $m\times n$ matrices $B$ and $D$ of the form $$
B=\begin{bmatrix*}
1 & 1 & \cdots  & 1\\
x_1^{d_1-1} & x_2^{d_{1}-1} & \cdots &x_n^{d_{1}-1} \\  
\vdots &\vdots&\ddots&\vdots\\       
  x_1^{d_{m-1}-1} & x_2^{d_{m-1}-1} & \cdots & x_n^{d_{m-1}-1}\\   
\end{bmatrix*},
\;\;
D=\begin{bmatrix*}
1 & 1 & \cdots  & 1\\
x_1^{d_1} & x_2^{d_{1}} & \cdots &x_n^{d_{1}} \\  
\vdots &\vdots&\ddots&\vdots\\       
  x_1^{d_{m-1}} & x_2^{d_{m-1}} & \cdots & x_n^{d_{m-1}}\\   
\end{bmatrix*}
$$

Let $\beta_{\Lambda}$ be an $m\times m$ minor of $B$ involving columns $\Lambda$ where $\Lambda\subset [n]$ and $|\Lambda|=m$. Let $J_{m,n}:=J_{m,n}(\underline d)=\langle\delta_{i_1,\ldots,i_m}|1\leq i_1<\ldots<i_m\leq n\rangle$ where $\delta_{i_1,\ldots,i_{m}}$ is an $m\times m$ minors of $D$ and $J:=J(\underline d)=I_{m,n}+J_{m,n}(\underline d)$. 
\end{setup} 

One can observe the relations between $m\times m$ minors of $B$ and $D$ as given in the remark below. 
\begin{remark}\label{minBC}
With notation in Setup \ref{Setup},  we see that $\delta_{\Lambda}=\beta_{\Lambda}\sum_{\tau\subset \Lambda,|\tau|=m-1}x_{\tau}$ in $R/I_{m,n}$. 
\end{remark} 

The following lemma is crucial in the proof of Theorem \ref{inj}.
\begin{lemma}\label{Minormap}
Assume Setup \ref{Setup}. Let $\sigma \subset [n]$, $|\sigma|=m-1$, and $f_{\sigma}:\omega_{R/I_{m,n}}\rightarrow R/I_{m,n}$ be the map defined in Theorem \ref{Mingens}. Let $\psi:\omega_{R/I_{m,n}}\rightarrow R/I_{m,n}$ be the map given by 

$$\psi=\sum\limits_{\substack{\Lambda\subset [n]\\|\Lambda|=m}}\;\;\sum\limits_{\substack{\sigma\subset \Lambda\\|\sigma|=m-1}}\beta_{\Lambda}f_{\sigma}.$$
Then $J/I_{m,n}\subset \im(\psi)$.
\end{lemma}

\begin{proof}
Let $1<j_1<\ldots <j_{m-1}\leq n$. For all $\Lambda_1,\Lambda_2\subset [n]$ and $|\Lambda_1|=|\Lambda_2|=m$, we consider $\sigma=\Lambda_1\cap\Lambda_2 $ with $|\sigma|=m-1$. If $\sigma\subset \{1,j_1,\ldots,j_{m-1}\}$, then we see that $$\beta_{\Lambda_1}x_{\sigma}=\beta_{\Lambda_1}f_{\sigma}([T_{[n]\setminus \{j_1,\ldots,j_{m-1}\}}])=\beta_{\Lambda_2}f_{\sigma}([T_{[n]\setminus \{j_1,\ldots,j_{m-1}\}}])=\beta_{\Lambda_2}x_{\sigma}$$ in $R/I_{m,n}$. Hence, $$\psi([T_{[n]\setminus\{j_1,\ldots,j_{m-1}\}}])=(m-1)\beta_{\Lambda}\sum\limits_{\substack{\tau\subset \Lambda\\ |\tau|=m-1}}f_{\tau}([T_{[n]\setminus \{j_1,\ldots,j_{m-1}\}}])$$ where $\{j_1,\ldots,j_{m-1}\}\subset \Lambda$. 
By Theorem \ref{Mingens}, $f_{\tau}([T_{[n]\setminus \{j_1,\ldots,j_{m-1}\}}])=x_{\tau}$ provided $\tau\subset \Lambda$ and $|\tau|=m-1$. Thus, we get $\psi([T_{[n]\setminus\{j_1,\ldots,j_{m-1}\}}])=(m-1)\delta_{\Lambda}$ by Remark \ref{minBC}. Hence, for every $\Lambda\subset [n]$, we have $\delta_{\Lambda}\in \im(\psi)$. This proves $J/I_{m,n}\subset \im(\psi)$. 
\end{proof}

We are now ready to state and prove  the main theorem in this paper.
\begin{theorem}\label{inj}
 With notation in Setup \ref{Setup}, let $\psi:\omega_{R/I_{m,n}}\rightarrow R/I_{m,n}$ be the map stated in Lemma \ref{Minormap}. Then $\psi$ is injective and $\im(\psi)=J/I_{m,n}$.
\end{theorem}
\begin{proof}  

Let $S=R/I_{m,n}$. Consider a short exact sequence of the form 
\begin{equation}\label{ses2}
0\rightarrow\ker(\psi)\rightarrow \omega_{S}(-d)\xrightarrow{\psi} S \rightarrow S/\im(\psi)\rightarrow 0. 
\end{equation}
By Lemma \ref{Minormap}, we get $J/I_{m,n}\subset \im(\psi)$, and hence $H_{S/\im(\psi)}(t)\leq H_{R/J}(t)$. Set  $P_{m,n}^k$  as 
\begin{align*}
P_{m,n}^k &=\langle {\rm in}(x_{n-k+1}x_{n-k}\ldots x_{n-1}x_n\delta_{i_1,i_2,\ldots,i_{m-k-1},n-k,n-k+1,\ldots,n})|1\leq i_1<\ldots<i_{m-k-1}<n-k-1\rangle\\
&=\langle x_{i_1}^{d_{m-1}}\ldots x_{i_{m-k-1}}^{d_{k+1}}x_{n-k+1}^{d_{k}+1}\ldots x_{n}^{d_1+1}|1\leq i_1<i_2<\ldots<i_{m-k-1}<n-k-1\rangle.
\end{align*}  
Then $Q=I_{m,n}+\sum_{k=0}^{m-1}P^k_{m,n}$ is an ideal of $R$ such that $Q\subset {\rm in }(J)$. Futhermore, the fact that $H_{R/J}(t)=H_{R/{\rm in}(J)}(t)$ implies  
\begin{equation}\label{ineq}
H_{S/\im(\psi)}(t)\leq H_{R/J}(t)\leq H_{R/Q}(t).
\end{equation}
To prove the theorem, it is enough to see that
$H_{R/Q}(t)\leq H_{S/\im(\psi)}(t)$.

Let $A_{i_1,\ldots, i_{m-k-1},k}=\sk[x_{i_1},\ldots,x_{i_{m-k-1}},x_{n-k+1},\ldots,x_n]$ and let the $\sk$-linear maps $g_{i_1,\ldots,i_{m-k-1}}:A_{i_1,\ldots,i_{m-k-1},k}\rightarrow Q/I_{m,n}$ be defined as
\begin{align*}
g_{i_1,\ldots,i_{m-k-1},k}(1)&=x_{i_1}^{d_{m-1}}\ldots x_{i_{m-k-1}}^{d_{k+1}}x_{n-k+1}^{d_{k}+1}\ldots x_{n}^{d_1+1}. 
\end{align*}
By the universal property of coproduct, we have 
$$Q/I_{m,n}\simeq \bigoplus_{1\leq i_1<\ldots <i_{m-k-1}<n-k-1,k=0}^{m-1}A_{i_1,\ldots,i_{m-k-1},k}(-d-k).$$

Hence, $H_{Q/I_{m,n}}(t)=t^{d}\dfrac{\sum_{k=0}^{m-1}\alpha_k t^k}{(1-t)^{m-1}}$ where $\alpha_k=\binom{n-k-1}{m-k-1}$. Therefore, by Lemma \ref{Omega},  we get $H_{Q/I_{m,n}}(t)=t^d H_{\omega_S}(t)$. By Equation \ref{ses2}, one can see that $$H_{R/Q}(t)=H_{S}(t)-H_{Q/I_{m,n}}(t)\leq H_{S/\im(\psi)}(t).$$ Then, by Equation \ref{ineq}, we have $H_{S/\im(\psi)}(t)= H_{R/J}(t)$, hence $\im(\psi)=J/I_{m,n}$ and $Q={\rm in}(J)$. Moreover, we get $$H_{R/J}(t)=H_{S/\im(\psi)}(t)=H_{S}(t)-t^d H_{\omega_{S}}(t).$$
By Equation \ref{ses2}, $H_{\ker(\psi)}(t)=0$, hence $\ker(\psi)=0$. Thus, $\psi$ is injective.
\end{proof}

In the following corollary, we see that $\omega_{R/I_{m,n}}$ is identified with an $\mathbb{N}$-graded ideal of $R/I_{m,n}$ for each $1<d_1<\ldots<d_{m-1}$.
\begin{corollary}\label{omg}
Assume Setup \ref{Setup}. Then $\omega_{R/I_{m,n}}(-d)\simeq J/I_{m,n}$. 
\end{corollary}
\begin{proof}
The following diagram is commutative:
\[
\xymatrixrowsep{1.8pc} \xymatrixcolsep{2.2pc}
\xymatrix{
0\ar@{->}[r]^{}&J/I_{m,n}\ar@{->}[r]^{}&R/I_{m,n}\ar@{->}[r]^{}&R/J\ar@{->}[r]^{}&0\\
0\ar@{->}[r]^{}&\omega_{R/I_{m,n}}(-d)\ar@{->}[r]^{\psi}\ar@{->}[u]^{\psi}& R/I_{m,n}\ar@{->}[u]^{\simeq}\ar@{->}[r]^{}&R/J\ar@{->}[r]^{}\ar@{->}[u]^{\simeq}&0\\
}
\]
By the Snake Lemma, $\omega_{R/I_{m,n}}(-d)\simeq J/I_{m,n}$.
\end{proof}

\begin{corollary}\label{Gor1}
There are infinitely many $\mathbb{N}$-graded embeddings of $\omega_{R/I_{m,n}}$ into $R/I_{m,n}$.
\end{corollary}
As a consequence of Theorem~\ref{inj}, each specific embedding of canonical module $\omega_{R/I_{m,n}}$ to $R/I_{m,n}$ produces a Gorenstein ring.

\begin{proposition}
Let $J=I_{m,n}+J_{m,n}$ and $d=d_1+\ldots+d_{m-1}$. Then $R/J$ is a Gorenstein ring with $\dim(R/J)=m-2$.    
\end{proposition}
\begin{proof}
Let $\psi:\omega_{R/I_{m,n}}\rightarrow R/I_{m,n}$ be the map in Theorem \ref{inj}. Then ${\rm Cone}(\psi)$ is a minimal free resolution of $R/J$ with ${\rm Cone}(\psi)_{i}=(G_{n-m-i+1}^{m,n})\oplus F_{i}^{m,n}$, where $\mathbb{F}^{m,n}_{\bullet}$ and $\mathbb{G}^{m,n}_{\bullet}$ are minimal free resolutions, given in Proposition \ref{Gal} and Remark \ref{MinOmg}, of $R/I_{m,n}$ and $\omega_{R/I_{m,n}}$, respectively. Then we have $\pdim(R/J)=n-m+2$.

By Corollary \ref{omg}, we have $\omega_{R/I_{m,n}}(-d)\simeq J/I_{m,n}$, and hence $J/I_{m,n}$ is an ideal with finite resolution. Now note that $J/I_{m,n}$ contains a $R$-regular element by \cite[Corollary 1.4.7]{BH}. Since $\dim(R/I_{m,n})=m-1$,  we have $\dim(R/J)\leq m-2$. By the graded Auslander-Buchsbaum formula, we get $\depth(R/I_{m,n})=m-2$. Therefore $R/J$ is Cohen-Macaulay. Proposition \ref{Gal} implies that $\beta_i(R/I_{m,n})=\binom{n}{m+i}\binom{m+i-1}{i}$, and hence $\beta_{i}(R/J)=\beta_{n-m+2-i}(R/J)$. This proves that $R/J$ is Gorenstein.   
\end{proof}

\section{$\mathbb{N}^n$-Graded Embedding and Connected Sums}\label{sec6}

In this section we specialize to $m=2$. In this situation we define even more embeddings of $\omega_{R/I_{2,n}}$ in $R/I_{2,n}$ and all of these embeddings are even $\mathbb{N}^n$-graded. These embeddings are closely related to connected sums of several copies of certain Artinian  $\sk$-algebras. 

Throughout the rest of the section we fix an $n$-tuple $(e_1,\ldots ,e_n)$ of  integers bigger than $1$.
Let  $e=e_1\ldots e_n$ and  $r_i=e_1\ldots \hat{e}_i\ldots e_n$ where $\hat{e}_i$ denotes the missing term in the product. For $2\leq i\leq n$ we set $\deg(x_i)=r_i$
which makes $R$ an $\mathbb{N}^n$-graded ring. In this setup we have the following.

\begin{theorem}\label{m=2}
The map $\psi:\omega_{R/I_{2,n}}\rightarrow R/I_{2,n}$  defined by $$\psi([T_{[n]\setminus \{i\}}])=x_i^{e_i}-x_1^{e_1}$$
is an $\mathbb{N}^n$-graded embedding. 
\end{theorem}
\begin{proof}
By Theorem \ref{Mingens}, $\{f_{\{i\}}:\{i\}\subset [n]\}$ is a generating set of $\Hom_R(\omega_{R/I_{2,n}},R/I_{2,n})$ where $f_{\{i\}}([T_{[n]\setminus \{i\}}])=x_i$ and $f_{\{1\}}([T_{[n]\setminus\{i\}}])=x_1$ for $i\neq 1$. Let $A$ be a $2\times n$ matrix of the form $$A=\begin{bmatrix*}
1 & 1 & \cdots  & 1\\
x_1^{e_1-1} & x_2^{e_{2}-1} & \cdots &x_n^{e_{n}-1}   
\end{bmatrix*}$$
and $\alpha_{1i}=x_i^{e_i-1}-x_1^{e_1-1}$ be a $2\times 2$ minor of $A$ for each $i$. Then, for $i\neq 1$, $$\psi([T_{[n]\setminus \{i\}}])=\alpha_{1i}(f_{\{i\}}-f_{\{1\}})([T_{[n]\setminus \{i\}}])=x_i^{e_i}-x_1^{e_1}.$$
Now consider ideals of the form $L_{2,n}:=L_{2,n}(e_1, \ldots, e_n)=\langle x_i^{e_i}-x_1^{e_1}:2\leq i\leq n\rangle$ and $L=I_{2,n}+L_{2,n}$. There is a short exact sequence of the form
\begin{equation}\label{Ses1}
0\rightarrow\ker(\psi)\rightarrow \omega_{R/I_{2,n}}(-e)\xrightarrow{\psi}R/I_{2,n}\rightarrow R/L\rightarrow 0.
\end{equation}

Next we show that $\tilde{H}_{\ker(\psi)}(t)=0$. By Remark \ref{Poin1}, $$\displaystyle{\tilde{H}_{R/I_{2,n}}(t)-\tilde{H}_{R/J}(t)=-t^e\Big[1+\sum_{i=1}^n\dfrac{1}{t^{r_i}-1} \Big]=-t^e \tilde{H}_{R/I_{2,n}}(t^{-1})}.$$
By \cite[Corollary. 4.4.6]{BH}, we get $\tilde{H}_{\omega_{I_{2,n}}}(t)=-\tilde{H}_{R/I_{2,n}}(t^{-1})$,  hence $$\displaystyle{\tilde{H}_{R/I_{2,n}}(t)-\tilde{H}_{R/L}(t)=t^e \tilde{H}_{\omega_{I_{2,n}}}(t)}.$$ The short exact sequence in (\ref{Ses1}) implies $\tilde{H}_{\ker(\psi)}(t)=0$. Therefore $\ker(\psi)=0$. This proves $\psi$ is injective. 
\end{proof}
As an immediate consequence of Theorem \ref{m=2},  $\omega_{R/I_{m,n}}$ is identified with an $\mathbb{N}^n$-graded ideal of $R/I_{m,n}$ for a specific embedding.
\begin{corollary}\label{m=2Omega}
With notation as in Theorem \ref{m=2}, let $L_{2,n}=\langle x_i^{e_i}-x_1^{e_1}:2\leq i\leq n\rangle$ and $L=I_{2,n}+L_{2,n}$. Then $\omega_{R/I_{2,n}}(-e)\simeq L/I_{2,n}$. 
\end{corollary}
\begin{proof}
By Theorem \ref{m=2}, the following commutative diagram is 
\[
\xymatrixrowsep{1.8pc} \xymatrixcolsep{2.2pc}
\xymatrix{
0\ar@{->}[r]^{}&L/I_{2,n}\ar@{->}[r]^{}&R/I_{2,n}\ar@{->}[r]^{}&R/L\ar@{->}[r]^{}&0\\
0\ar@{->}[r]^{}&\omega_{R/I_{2,n}}(-e)\ar@{->}[r]^{\psi}\ar@{->}[u]^{\psi}& R/I_{2,n}\ar@{->}[u]^{\simeq}\ar@{->}[r]^{}&R/L\ar@{->}[r]^{}\ar@{->}[u]^{\simeq}&0\\
}
\]
By the Snake Lemma, $\omega_{R/I_{2,n}}(-e)\simeq L/I_{2,n}$.
\end{proof}

Now we state an application of Theorem \ref{m=2} which gives a minimal free resolution of a connected sum of Artinian rings of embedding dimension one.
\begin{corollary}\label{Connectedsum}
Let $R_i=\sk[x_i]/\langle x_i^{e_i+1}\rangle$, $L_{2,n}=\langle x_i^{e_i}-x_1^{e_1}:2\leq i\leq n\rangle$, and $L=I_{2,n}+L_{2,n}$. Suppose $\psi:\omega_{R/I_{2,n}}\rightarrow R/I_{2,n}$ satisfies Theorem \ref{m=2}.  Then $R/L$ is a Gorenstein Artin ring such that $R/L\simeq R_1\#_{\sk}\ldots\#_{\sk}R_n$. Furthermore, the mapping cone ${\rm Cone}(\psi)$ is a minimal free $R$-resolution of $R/L$.
\end{corollary}
\begin{proof}
First note that $\soc(R_i)=\langle x_i^{e_i}\rangle$. By Definition \ref{CSdef}, we have  $R_1\#_{\sk}\ldots\#_{\sk}R_{n}\simeq R/L$, and hence $R/L$ is Gorenstein by Remark \ref{CSrmk}. Using Corollary \ref{m=2Omega}, we get $\omega_{R/I_{m,n}}(-e)\simeq L/I_{m,n}$. Let $\mathbb{F}^{m,n}_{\bullet}$ and $\mathbb{G}^{m,n}_{\bullet}$ be minimal free resolutions of $R/I_{m,n}$ and $\omega_{R/I_{m,n}}$, respectively as in Proposition \ref{Gal} and Remark \ref{MinOmg}. Then ${\rm Cone}(\psi)$ is a minimal free resolution of $R/L$.

\end{proof}


\begin{thebibliography}{11}

\bibitem{ACJY} H.~Ananthnarayan, E.~Celikbas, Jai Laxmi, Z.~Yang, \emph{Decomposing Gorenstein rings as connected sums}, arXiv:1406.7600.
\bibitem{BH} W.~Bruns~and~ J.~Herzog, \emph{Cohen-Macaulay rings}, Cambridge Studies in Advanced Mathematics, vol. 39, Cambridge University Press, Cambridge, 1993.
\bibitem{FW} W.~Fulton, \emph{Young tableaux}, London Mathematical Society Student 
Texts, vol. 35, Cambridge University Press, Cambridge, 1997.
\bibitem{GF} F.~Galetto, \emph{On the Ideals Generated by all Squarefree Monomials of a Given Degree}, arXiv:1609.06396. 
\bibitem{JM} G.~D.~James, \emph{The Representation Theory of the Symmetric Groups}, vol. 682, Lectures Notes in Mathematics, Springer, Berlin.





\end{thebibliography}
\end{document}